\documentclass[12pt]{article}
\usepackage{amsmath}
\usepackage{amsfonts}
\usepackage[usenames,dvipsnames,svgnames,table]{xcolor}
\usepackage{amssymb}
\usepackage{latexsym}
\usepackage{mathrsfs}
\usepackage{amssymb}
\usepackage{amsthm}
\usepackage[makeroom]{cancel}

\def\h{\hbox}
\def\<{\langle}
\def\>{\rangle}
\numberwithin{equation}{section}
\def\<{\langle}
\def\>{\rangle}

\def\ra{\rightarrow}
\def\p{\partial}
\def\a{\alpha}

\def\ld{{\lambda}}

\def\ov{\overline}

\def\h{\hbox}

\def\b{\beta}
\def\a{\alpha}
\def\G{\Gamma}

\newcommand{\CC}{{\mathbb C}}

\newcommand{\RR}{{\mathbb R}}

\newcommand{\NN}{{\mathbb N}}

\newcommand{\wt}{\widetilde}

\newtheorem{thm}{Theorem}[section]

\newtheorem{cor}[thm]{Corollary}

\newtheorem{defn}[thm]{Definition}
\newtheorem{conj}[thm]{Conjecture}
\newtheorem{example}[thm]{Example}

\numberwithin{equation}{section}
\begin{document}
\date{}
\author{Xiaojun Huang \footnote{Supported in part by  DMS-2247151 and DMS-2000050}\  \
and Wanke Yin\footnote{Supported in part by  NSFC-12171372}}

\title{\bf Regular types and order of vanishing along  a set of non-integrable vector fields}

\vspace{3cm} \maketitle \centerline{ Dedicated to the memory of
Professor Zhi-Hua Chen}
\medskip

\vspace{3cm} \maketitle

{\bf Abstract}: This paper has two parts. We first survey  recent
efforts on the   Bloom conjecture \cite{Bl2} which  still remains
open in the case of complex dimension at least 4.  Bloom's
conjecture concerns the equivalence of three regular types. There is
a more general  important   notion, called the singular D'Angelo
type (or simply, D'Angelo type) \cite{DA1}. While the finite
D'Angelo type condition is the right one for the study of local
subelliptic estimates for Kohn's $\ov{\p}$-Neumann problem, regular
types are important as their finiteness gives the global regularity
up to the boundary of solutions of Kohn's $\ov{\p}$-Neumann problem
\cite{Ca2} \cite{Zai}.

In the second part of the paper, we provide a proof of a seemingly
elementary but  a truly fundamental property (Theorem \ref{cor1.3}
or its CR version Theorem \ref{cor2.3}) on the vanishing order of
smooth functions along a system of non-integrable vector fields.
A special case, Corollary \ref{cor2.4}, of Theorem \ref{cor2.3} had
 already appeared in a paper of D'Angelo [pp 105, 3. Remark,
\cite{DA3}].
A main goal in this part is to provide   proofs for these results
for the purpose of future references. Our arguments are based on a
deep normalization theorem for a system of non-integrable vector
fields due to Helffer-Nourrigat \cite{HN}, as well as its late
generalization in Boauendi-Rothschild \cite{BR}.

\section{Regular types and Bloom conjecture}

\subsection{Introduction}

For  a smoothly bounded pseudoconvex domain $D$ in  ${\mathbb C}^n$
with $n\ge 2$, many analytic and geometric properties of  $D$ are
determined by its invariants from the inherited CR structure bundle
over $\p D$. In the 1960's, Kohn \cite{FK} established the
subelliptic estimate for the $\ov{\p}$-Newmann problem when the Levi
form of $\p D$ is positive definite everywhere, which is called the
strong pseudoconvexity of $D$.

To generalize his subelliptic estimate for the $\ov\partial$-Neumann
problem to bounded weakly pseudoconvex domains in ${\mathbb C}^2$,
Kohn in his fundamental paper \cite{Kohn1} introduced three
different boundary CR invariants for $D\subset {\mathbb C}^2$. These
invariants are, respectively, the maximum order of contact at $p\in
\p D$  with smooth holomorphic curves at  $p$, denoted by
$a^{(1)}(\p D, p)$ and called the contact type at $p$, order of
vanishing added by two of  the Levi-form along the contact bundle,
denoted by $c^{(1)}(\p D, p)$ and called the Levi-form type at $p$,
and the length of the iterated Lie brackets of boundary CR vector
fields as well as their conjugates needed to recover the boundary
contact direction, denoted by $t^{(1)}(\p D, p)$ and called vector
field commutator type at $p$. Kohn proved that all these invariants
are in fact the same, thus simply called the type value of $\partial
D\subset {\mathbb C}^2$ at $p$. When this type value is finite at
each point, $D$ is called a smoothly bounded pseudoconvex domain of
finite type. Kohn's work in \cite{Kohn1}  shows that the subelliptic
estimate for $\ov\partial$-Neumann problems   holds for such a
domain. Together with that of Greiner \cite{Gr}, Kohn's work also
gives the precise information of the subelliptic gain for the
$\ov\partial$-Neumann problem for a smoothly bounded finite type
weakly pseudoconvex domain in ${\mathbb C}^2$.


Generalizations of Kohn's notion of the boundary finite type
condition to higher dimensions have been a subject under  extensive
investigations in the past  half century in Several Complex
Variables.

Kohn later introduced a  finite type condition in higher dimensions
through the subelliptic multipliers  \cite{Kohn2}. The understanding
of this type  has later revived to be a very active field of studies
through the work of many people including
Diederich-Fornaess \cite{DF}, Siu \cite{Siu1},
 Kim-Zaistev \cite{KZ1} \cite{KZ2}, Nicoara \cite{Nic}  as well as  the reference therein.

Bloom \cite{Bl1} and Bloom-Graham \cite{BG1} established Kohn's
original notion of types in ${\mathbb C}^2$ to any dimensions.
Namely, for each integer $s\in [1,n-1]$ and for a smooth real
hypersurface $M\subset {\mathbb C}^n$ with $n\ge 2$  and $p\in M$,
Bloom-Graham and Bloom defined the vector field commutator type
$t^{(s)}(M,p)$ , the Levi-form type $c^{(s)}(M,p)$ and the regular
contact type $a^{(s)}(M,p)$ of $M$ at $p$, which are called the
regular multi-types of $M$ at $p\in M$.
Bloom-Graham \cite{BG1} and Bloom \cite{Bl1} showed that when
$s=n-1$, all these types are also the same as in the case of $n=2$.
However, without pseudoconvexity for $M$, Bloom \cite{Bl2} showed
that when $s\not = n-1$, while the contact type $a^{(s)}$ may be
finite, the commutator type $t^{(s)}$ and the Levi-form type
$c^{(s)}$ can be infinite in many examples. The commutator type is
intrinsically defined only through the Lie bracket of CR or
conjugate CR vector fields of $M$ valued in some smooth subbundles
of $T^{(1,0)}M\oplus T^{(0,1)}M$. This notion has already been an
important object in the fields such as  Sub-elliptic Analysis and
Partial Differential Equations. It is often referred as H\"ormander
commutator type in the literature. The other two types are more on
the emphasis of complex analysis, defined through the complex
structure of the ambient complex space.

Different from the case of complex dimension two, the regular types
are not the right one  for  the  study of the local subelliptic
estimates for the $\ov\partial$-Neumann problem for domains in a
complex space of complex dimension at least three. However, the
early work of Catlin \cite{Ca2} and the very recent  nice work of
Zaitsev \cite{Zai} showed that  finite regular types force the
finiteness of Catlin's multitypes and also Zaitsev's tower
multi-types (which are similarly defined as the Levi-form type but
only for a certain formally integrable smooth subbundles of
$T^{(1,0)}$). It then can be used to produce a stratification of $\p
D$ into submanifolds with controlled holomorphic dimension, which
provides Property-P for $\p D$ with a finite regular type condition.
Therefore the finite regular type of $\p D$ gives the compactness of
the Newmann operator of $D$, and thus the global regularity of the
$\ov\partial$-Neumann problem follows from the classical work of
Kohn-Nirenberg \cite{Str}.


The fundamental work of D'Angelo in \cite{DA1} studied the  notion
of singular types, widely called the D'Angelo types, by considering
the order of contact with not just smooth complex submanifolds but
possibly singular complex analytic varieties. D'Angelo finite type
condition  is a singular contact type condition. Its significance is
its equivalence to the existence of the local subelliptic estimate
after the fundamental   work of Kohn \cite{Kohn2},
Diederich-Fornaess \cite{DF} and Catlin \cite{Ca1}.

The types   mentioned above were introduced through different
aspects of studies. Revealing the connections among them always
results in a deeper understanding of the subject. For instance,
proving that the Kohn multiplier ideal  type is equivalent to the
finite D'Angelo type would provide a new and  more direct solution
of the $\ov\partial$-Neumann problem.

\subsection{Bloom's  conjecture and D'Angelo's conjecture}\label{subb1}

Let $M\subset \mathbb{C}^n$ be a smooth real hypersurface with $p\in
M$. Then  $T^{1,0}M$ is a  smooth smooth vector bundle over $M$ of
complex dimension $(n-1)$. A smooth section $L$ of  $T^{1,0}M$ is
called a smooth vector field of type $(1,0)$  or a CR vector field
along $M$, and its complex conjugate is called smooth vector field
of type $(0,1)$ or a conjugate CR vector field along $M$. Let $\rho$
be a defining function of $M$, namely, $\rho\in C^\infty(U)$ with
$U$ an open neighborhood of $M\subset\mathbb{C}^n$ and   $U\cap
M=\{\rho=0\}\cap U$, $d\rho|_{U\cap M}\neq 0$. Denote by ${\mathcal
X}_{\CC}(M)$  the $C^\infty(M)$-module of all
 complex valued smooth vector fields tangent to $M$.
 Write $\theta=-\frac{1}{2}(\p \rho-
\ov{\p}\rho)$, called a pure imaginary contact form along $M$. Write
$T= 2i\h{Im}\left(\sum_{j=1}^{n}\frac{\p
\rho}{\p\ov{z_j}}\frac{\p}{\p z_j}\right),$ called a pure imaginary
contact vector field of $M$. The following holds trivially:
 $$\langle \theta,T\rangle=|\p \rho|^2> 0,\  \langle L,\theta\rangle=0\
 \h{ for  any } \h{CR vector field }
 \ L \h{ along }M.$$
 For two tangent vector fields $X, Y\in {\mathcal X}_{\CC}(M)$,
 define
 the Levi form $$\lambda(X,Y)=\langle\theta, [X,\ov{Y}]\rangle.$$ By the
Cartan lemma, $\lambda(X,Y)=2\langle d\theta,
X\wedge\ov{Y}\rangle=d\theta(X,Y).$ After replacing $\rho$ by
$-\rho$, when needed, if we can make $\lambda(L, L)$ positive
definite along $M$ for any vector field $L\not=0$ of type $(1,0)$,
we say $M$ is strongly pseudoconvex. If we can only make $\ld$
semi-positive definite, we call $M$ a weakly pseudoconvex
hypersurface. When $L_1=\sum_{j=1}^{n}\xi_j\frac{\p }{\p z_j},
L_2=\sum_{j=1}^{n}\eta_j\frac{\p }{\p z_j}$, we then have
\begin{equation}\label{004}
\ld(L_1,L_2)=\sum_{j,\ell=1}^{n}\frac{\p^2\rho}{\p z_j \ov{\p}
z_\ell}\xi_j\ov{\eta_\ell}.
\end{equation}
Levi form is a Hermitian form over $T^{(1,0)}M$.

For any $1\leq s\leq n-1$, let $B$ be a smooth complex vector
subbundle of $T^{1,0}M$  of complex dimension $s$. Let
$\mathcal{M}_1(B)$ be the $C^\infty(M)$-submodule of ${\mathcal
X}_\CC(M)$ spanned by the smooth $(1,0)$ vector fields $L$ with
$L|_q\in B|_q$ for each $q\in M$, together with their complex
conjugates. For $\mu\geq 1$, we let $\mathcal{M}_\mu(B)$ denote the
$C^\infty(M)$-submodule spanned by commutators of length less than
or equal to $\mu$ of vector fields from $\mathcal{M}_1(B)$
(including $\mathcal{M}_1(B)$). Here, a commutator of length $\mu\ge
2$ of vector fields in $\mathcal{M}_1(B)$ is a vector field of the
following form: $[Y_{\mu},[Y_{\mu-1},\cdots,[Y_2,Y_1]\cdots]$ with
$Y_j\in \mathcal{M}_1(B)$. Define $t^{(s)}(B,p)=m$ if $\langle F,\p
\rho\rangle(p)=0$ for any $F\in \mathcal{M}_{m-1}(B)$ but $\langle
G,\p \rho\rangle(p)\neq 0$ for a certain $G\in \mathcal{M}_{m}(B)$.
Namely, $m$ is the smallest number such that
$\mathcal{M}_m(B)|_p\not\subset T^{(1,0)}_p\p D\oplus T^{(0,1)}_p\p
D.$ If such an $m$ does not exist, we set $t^{(s)}(B,p)=\infty.$
When $B$ has complex dimension one with $B$ being spanned by a
$(1,0)$-type vector field $L$ near $p$, we also write
$t^{(1)}(B,p)=t_L(\p D, p).$  $t^{(s)}(B,p)$ is called the vector
field commutator type of $B$ at $p$, or the commutator type of $L$
at $p$ when $B_q=span\{L|_q\}$ for $q$ near $p$. Define
\begin{equation}\begin{split}
  t^{(s)}(M,p)=\sup\limits_{B}\{t(B,p)|\ B\ \text{is an
   $s$-dimensional subbundle of\ }\ T^{1,0}M\}.
 \end{split} \end{equation}

$t^{(s)}(M,p)$ is called the  $s^{th}$-vector field commutator type
of $M$ at $p$, or simply the  $s^{th}$ commutator type of $M$ at
$p$.

\medskip
Write $\Gamma_{\infty}(B)$ for the set of  smooth sections of $B$.
We define $c^{(s)}(B,p)=m$ if for any $m-3$ vector fields
$F_1,\cdots,F_{m-3}$ of $\mathcal{M}_1(B)$ and any $L\in
\Gamma_{\infty}(B)$ with $L_p\not =0$, it holds that
$$
F_{m-3}\cdots F_{1}\big(\ld(L,L)\big)(p)=0;
$$
and for a certain choice of $m-2$ vector fields $G_1,\cdots,G_{m-2}$
of $\mathcal{M}_1(B)$ and a certain  $L\in  \Gamma_{\infty}(B)$ with
$L_p\not =0$, we have
$$
G_{m-2}\cdots G_{1}\big(\ld(L,L)\big)(p)\neq 0.
$$
When such an $m$ does not exist, we then set $c^{(s)}(B,p)=\infty$.
We define
\begin{equation}\begin{split}
  c^{(s)}(M,p)=\sup\limits_{B}\{c^{(s)}(B,p): \  B\ \text{ is an $s$-dimensional subbundle of} \ T^{1,0}M
  \}.
 \end{split} \end{equation}
We call $ c^{(s)}(B,p)$ the Levi-form type of $B$ at $p$ and
$c^{(s)}(M,p)$ the $s$-Levi form type of $\p D$ at $p$. When such an
$m$ does not exist, we define $c^{(s)}(M,p)=\infty$. For  an $L\in
\Gamma_{\infty}(B)$ with $L_p\not =0$, we similarly have the notion
of $c_{L}(M,p)$.

\medskip
We finally define the  $s$- regular contact type $a^{(s)}(M,p)$ as
follows:
 \begin{equation}\begin{split}
  a^{(s)}(M,p)=\sup\limits_{X}\big\{\ell|\ &\exists  \text{ an $s$-dimensional
  complex submanifold}\ X\\
  &\text{whose order of contact  with $M$ at $p$ is $\ell$}\big\}.
 \end{split} \end{equation}

Here we remark that the order of contact of $X$ with $M$ at $p$ is
defined as the order of vanishing of $\rho|_X$ at $p$.
\bigskip

In  \cite{Kohn1}, when $n=2$, Kohn showed that
$t^{(1)}(M,p)=c^{(1)}(M,p)=a^{(1)}(M,p)$. Bloom-Graham \cite{BG1}
and Bloom \cite{Bl1} proved   that for any smooth real hypersurface
$M\subset\CC^n$ with $p\in M$,
$$t^{(n-1)}(M,p)=c^{(n-1)}(M,p)=a^{(n-1)}(M,p).$$ And for any $1\leq s\leq n-2$, Bloom in \cite{Bl2}
observed that $a^{(s)}(M,p)\leq c^{(s)}(M,p)$ and $a^{(s)}(M,p)\leq
t^{(s)}(M,p)$. For  these results to hold there is no need  to
assume  the pseudoconvexity  of $M$. However, the following example
of Bloom shows that for $n\geq  3$, when $M$ is not pesudoconvex, it
may happen that $a^{(s)}(M,p)< c^{(s)}(M,p)$ and $a^{(s)}(M,p)<
t^{(s)}(M,p)$ for $1\leq s\leq n-2$.

\begin{example}[Bloom, \cite{Bl2}]\label{Bl-example} Let $\rho=2\text{Re}(w)+(z_2+\ov{z_2}+|z_1|^2)^2$
and let $M=\{(z_1,z_2,w)\in \mathbb{C}^3|\ \rho=0\}$. Let $p=0$.
Then $a^{(1)}(M,0)=4$ but $c^{(1)}(M,0)=t^{(1)}(M,0)=\infty$ .
\end{example}

To see the $M$ in Example \ref{Bl-example} is not pseudoconvex near
$p=0$, we notice that a real normal direction of $M$ near $0$ is
given by  $\frac{\p}{\p u}$ with $u=\h{Re}(w)$. Let $\psi(\xi):
\Delta\ra \CC^3$ be a holomorphic disk that is smooth up to the
boundary such that $\psi=(\psi^*,0)$ with $\psi^*$ being attached to
the Heisenberg hypersurface in $\CC^2$ defined by
$z_2+\ov{z_2}+|z_1|^2=0$ and $\psi(1)=0$. By the Hopf lemma for
pseudoconvex domains \cite{BER}, $\psi$ would have a non-zero
derivative along $u$-direction, which is a contradiction. Let
$L=\frac{\p}{\p z_1}-\ov{z_1}\frac{\p}{\p z_2}.$ Then $L$ is a CR
vector field along $M$. Notice that no matter how many times we
perform the Lie bracket for $L$ and $\ov{L}$, we get a vector field
without components in $\frac{\p}{\p w}$ and $\frac{\p}{\p \ov{w}}$,
we see that $t^{(1)}_L(M,p)=\infty$. One also computes by
(\ref{004}) that
$$\ld(L,L)=2(z_2+\ov{z_2})+2|z_1|^2.$$
Thus $L(\ld(L,L))=\ov{L}(\ld(L,L))\equiv 0.$ Hence,
$c_L(M,p)=\infty$.  (One can also conclude the non-pseudoconvexity
of $M$ near $0$ by (\ref{004}).) Thus
$t^{(1)}(M,0)=c^{(1)}(M,0)=\infty.$ $a^{(1)}(M,0)$   is at least 4
that is attained by the smooth holomorphic curve $(\xi,0,0)$.
Suppose $\phi(\xi)=(\phi_1(\xi),\phi_2(\xi),\phi_3(\xi))$ be a
smooth holomorphic curve with maximum order of contact with $M$ at
$0$. Assume $a^{(1)}(M,0)>4$. Then the vanishing order of
\begin{equation} \label{005}
\begin{split}h&=2\h{Re}(\phi_3)+(2\h{Re}(\phi_2)+|\phi_1|^2)^2=2\h{Re}(\phi_3)
+4\h{Re}^2(\phi_2)+
|\phi_1|^4+4\h{Re}(\phi_2)|\phi_1|^2\\
&=2\h{Re}(\phi_3+\phi_2^2)+4|\phi_2|^2+
|\phi_1|^4+2(\phi_2+\ov{\phi_2})|\phi_1|^2
\end{split}
\end{equation}
 is at least 5 at $0$. Hence $\phi_3=-\phi_2^2 \ \
 \h{mod}(|\xi|^5)$,
 $\phi_2=0\ \h{mod}(|\xi|^2)$ and
 $$h=4|\phi_2|^2+2(\phi_2+\ov{\phi_2})|\phi_1|^2+|\phi_1|^4=0\ \h{mod}(|\xi|^5).$$
 Since $\phi$ is a smooth curve, it apparently follows that
 $\phi_1=a_1\xi+O(\xi^2)$ with $a_1\not =0$ and
$\phi_2=a_2\xi^k+O(|\xi|^{k+1}$ with $k\ge 2, a_2\not =0.$ Comparing
coefficients of degree $4$ in $\xi$, we see a contradiction. Hence
$a^{(1)}(M,0)=4$, which is  achieved by a smooth holomorphic cure
$\phi(\xi)=(\xi,0,0)$.

With the pseudoconvexity assumption of  $M$, Bloom in \cite{Bl2}
showed that when $M\subset {\mathbb C}^3,$
$a^{(1)}(M,p)=c^{(1)}(M,p)$. Motivated by this result, Bloom in 1981
\cite{Bl2} formulated the following famous conjecture:

\begin{conj}\label {bloom-conj} Let $M\subset \mathbb{C}^n$ be a pseudoconvex real hypersurface with $n\geq 3$. Then for any $1\leq s\leq n-2$ and $p\in M$,
\begin{equation*}\begin{split}
t^{(s)}(M,p)=c^{(s)}(M,p)=a^{(s)}(M,p).
\end{split} \end{equation*}

\end{conj}

In a related work, D'Angelo in 1986 conjectured that under
pseudo-convexity assumption of $M$, one should have
$t^{(1)}_L(M,p)=c^{(1)}_L(M,p)$.

More generally, we formulate the following generalized D'Angelo
Conjecture:

\begin{conj}\label {d'angelo-conj} Let $M\subset \mathbb{C}^n$ be a
 pseudoconvex real hypersurface with $n\geq 3$.
 Then for any $1\leq s\leq n-2$, $p\in M$ and a smooth complex
 vector subbundle $B$ of $T^{(1,0)}M$, it holds that
$$t^{(s)}(B,p)=c^{(s)}(B,p).$$
\end{conj}

If confirmed to be true, the generalized D'Angelo conjecture would
imply $t^{(s)}(M,p)=c^{(s)}(M,p)$.

40 years after Bloom formulated his conjecture, it was completely
settled in the case of complex dimension three in \cite{HY}. More
generally, the following theorem was proved in \cite{HY}:
\begin{thm}\label{mainthm}\cite{HY}
Let $M\subset \mathbb{C}^n$ be a smooth pseudoconvex real
hypersurface with $n\geq 3$. Then for  $s= n-2$ and any $p\in M$, it
holds that
\begin{equation*}\begin{split}
t^{(n-2)}(M,p)=a^{(n-2)}(M,p)=c^{(n-2)}(M,p).
\end{split} \end{equation*}

\end{thm}

In particular,  we answered affirmatively  the Bloom conjecture in
the case of complex dimension three (namely, $n=3$):

\begin{thm}\label{thm-dim-three}\cite{HY} The Bloom conjecture holds in the case of complex dimension three.
 Namely, for   a smooth pseudoconvex real hypersurface $M\subset \mathbb{C}^3$  and $p\in M$, it holds that
\begin{equation*}\begin{split}
t^{(1)}(M,p)=a^{(1)}(M,p)=c^{(1)}(M,p).
\end{split} \end{equation*}
\end{thm}

The solution to D'Angelo's conjecture in the case of complex
dimension three was  obtained   in \cite{CYY}, fundamentally based
on  results obtained in \cite{HY}.

\begin{thm}\label{thm-dim-three-CYY}\cite{CYY}
 For   a smooth pseudoconvex real hypersurface $M\subset \mathbb{C}^3$
  and $p\in M$, for any smooth CR vector field $L$ along $M$ with $L|_p\not =0$, it holds that
$$t^{(1)}_L(M,p)=c^{(1)}_L(M,p).$$
\end{thm}

The following theorem of D'Angelo also provides a partial solution
to D'Angelo's original  conjecture:

\begin{thm}\cite{DA2}\label{Da10}
\noindent Let $M\subset\CC^n$ be    a smooth pseudoconvex real
hypersurface
  and $p\in M$.

\noindent (1).  For any smooth CR vector field $L$ along $M$ with
$L|_p\not =0$,
  it holds that $c_L(M,p)\le \max\{t_L(M,p),\ 2t_L(M,p)-6\}.$

\noindent (2). If either $c_L(M,p)=4$ or $t_L(M,p)=4$, then both are
$4$.
\end{thm}

A recent paper in Huang-Yin \cite{HY3} generalizes Theorem
\ref{Da10}(2) to the case when either  $c_L(M,p)=6$ or $t_L(M,p)=6$.

The following examples show that  the conclusion of the D'Angelo
Conjecture may not hold  when $M$ is not pseudoconvex.

\begin{example}\label{exa}
Let $M$ be a real hypersurface in $\mathbb{C}^3$ with a defining
function
$$
\rho:=-(w+\ov{w})+|z_1|^4+z_1\ov{z_2}+z_2\ov{z_1}.
$$
Suppose $L$ is the tangent vector field of type $(1,0)$ defined by
$$
L=\frac{\p}{\p z_1}-|z_1|^2\frac{\p}{\p
z_2}+(\ov{z_2}+z_1\ov{z_1}^2)\frac{\p}{\p w}.
$$
Then $t_L(M,0)=\infty$ but $c_L(M,0)=4$.
\end{example}

First, a direct computation shows that
$$L\rho=2z_1\ov{z_1}^2+\ov{z_2}-|z_1|^2\ov{z_1}-(\ov{z_2}+z_1\ov{z_1}^2)\equiv 0.$$
 Hence $L$
is indeed  a  CR vector field along $M$. Notice that
\begin{equation*} \begin{split}
\p \ov{\p}\rho&=4|z_1|^2dz_1\wedge d\ov{z_1}+dz_1\wedge
d\ov{z_2}+dz_2\wedge d\ov{z_1},\\
\ld(L,{L})&=4|z_1|^2-|z_1|^2-|z_1|^2=2|z_1|^2.
\end{split} \end{equation*}
Hence
$$
L\ov{L}\ld(L,{L})=2\neq 0.
$$
Hence, we conclude that $c_L(M,0)=4$. We next compute $t_L(M,0)$ as
follows:
\begin{equation*} \begin{split}
[L,\ov{L}]&=\Big[\frac{\p}{\p z_1}-|z_1|^2\frac{\p}{\p
z_2}+(\ov{z_2}+z_1\ov{z_1}^2)\frac{\p}{\p w},\frac{\p}{\p
\ov{z_1}}-|z_1|^2\frac{\p}{\p
\ov{z_2}}+({z_2}+\ov{z_1}{z_1}^2)\frac{\p}{\p \ov{w}}\Big]\\
&=-\ov{z_1}\frac{\p}{\p \ov{z_2}}+2z_1\ov{z_1}\frac{\p}{\p
\ov{w}}-|z_1|^2\frac{\p}{\p \ov{w}}+{z_1}\frac{\p}{\p
{z_2}}-2z_1\ov{z_1}\frac{\p}{\p {w}}+|z_1|^2\frac{\p}{\p {w}}
\end{split} \end{equation*}
Thus
\begin{equation*} \begin{split}
[L,[L,\ov{L}]]&=2\ov{z_1}\frac{\p}{\p \ov{w}}+\frac{\p}{\p
{z_2}}-2\ov{z_1}\frac{\p}{\p {w}}-\ov{z_1}\frac{\p}{\p
\ov{w}}+\ov{z_1}\frac{\p}{\p {w}}-(-\ov{z_1})\frac{\p}{\p {w}})\\
&=\ov{z_1}\frac{\p}{\p \ov{w}}+\frac{\p}{\p {z_2}}.
\end{split} \end{equation*}
Furthermore, we obtain
\begin{equation*} \begin{split}
[L,[L,[L,\ov{L}]]] &=\Big[\frac{\p}{\p z_1}-|z_1|^2\frac{\p}{\p
z_2}+(\ov{z_2}+z_1\ov{z_1}^2)\frac{\p}{\p w},\ \ov{z_1}\frac{\p}{\p
\ov{w}}+\frac{\p}{\p {z_2}}\Big]=0.\\
[\ov{L},[L,[L,\ov{L}]]] &=\Big[\frac{\p}{\p
\ov{z_1}}-|z_1|^2\frac{\p}{\p
\ov{z_2}}+({z_2}+\ov{z_1}{z_1}^2)\frac{\p}{\p \ov{w}},\
\ov{z_1}\frac{\p}{\p \ov{w}}+\frac{\p}{\p {z_2}}\Big]=\frac{\p }{\p
\ov{w}}-\frac{\p }{\p \ov{w}}=0.
\end{split} \end{equation*}
Notice that
$$
[\ov{L},[\ov{L},[L,\ov{L}]]]=-\ov{[L,[L,[L,\ov{L}]]] }=0,\ \
[\ov{L},[{L},[L,\ov{L}]]]=-\ov{[L,[\ov{L},[L,\ov{L}]]] }=0.
$$
Hence $t(L,0)=+\infty$. That $t_L(M,0)=\infty$ can also be seen
geometrically  as follows:

Notice that
$$
L(-2w+2z_1\ov{z_2}+|z_1|^4)=\ov{L}(-2w+2z_1\ov{z_2}+|z_1|^4)=0.
$$
Thus  Re$(L)$, Im$(L)$ as well as  their Lie brackets of any length
are all tangent to the real codimension two submanifold of $\CC^3$
defined by $\{(z_1,z_2,w)\in \mathbb{C}^3:\
2w=2z_1\ov{z_2}+|z_1|^4\}$. Hence the Lie brackets of Re$(L)$ and
Im$(L)$ of any length will always been annihilated by the contact
form $\theta|_0=-\frac{1}{2}({\p w}-{\p \ov{w}})|_0$ at $0$, which
shows that $t(L,0)=+\infty$.

Finally, we present the following example, in which $t_L(M,0)$ and
$c_L(M,0)$ are finite but different.
\begin{example}\label{exa1}
Let $M$ be a real hypersurface in $\mathbb{C}^3$ with a defining
function
$$
\rho:=-(w+\ov{w})+|z_1|^4+z_1\ov{z_2}+z_2\ov{z_1}+|z_2|^{2}.
$$
Suppose $L$ is the tangent vector field of type $(1,0)$ defined by
$$
L=\frac{\p}{\p z_1}-|z_1|^2\frac{\p}{\p
z_2}+(\ov{z_2}+z_1\ov{z_1}^2-|z_1|^2\ov{z_2})\frac{\p}{\p w}.
$$
Then $t_L(M,0)=6$ but $c_L(M,0)=4$.
\end{example}

As in Example \ref{exa}, we have $L\rho\equiv 0$ and thus $L$ is
indeed tangent to $M$. Next
\begin{equation*} \begin{split}
\p \ov{\p}\rho&=4|z_1|^2dz_1\wedge d\ov{z_1}+dz_1\wedge
d\ov{z_2}+dz_2\wedge d\ov{z_1}+dz_2\wedge d\ov{z_2},\\
\ld(L,{L})&=4|z_1|^2-|z_1|^2-|z_1|^2+|z_1|^4=2|z_1|^2+|z_1|^4.
\end{split} \end{equation*}
Hence
$$
L\ov{L}\ld(L,{L})(0)=2\neq 0.
$$
We conclude that $c_L(M,0)=4$. As in Example \ref{exa}, $t_L(M,0)$
can be derived as follows:
\begin{equation*} \begin{split}
[L,\ov{L}] =&-\ov{z_1}\frac{\p}{\p
\ov{z_2}}+2z_1\ov{z_1}\frac{\p}{\p \ov{w}}-|z_1|^2\frac{\p}{\p
\ov{w}}+{z_1}\frac{\p}{\p {z_2}}-2z_1\ov{z_1}\frac{\p}{\p
{w}}+|z_1|^2\frac{\p}{\p
{w}}\\
&+(-\ov{z_1}{z_2}+|z_1|^4)\frac{\p}{\p
\ov{w}}+(z_1\ov{z_2}-|z_1|^4)\frac{\p}{\p {w}}\\
=&-\ov{z_1}\frac{\p}{\p \ov{z_2}}+{z_1}\frac{\p}{\p
{z_2}}-(|z_1|^2-z_1\ov{z_2}+|z_1|^4)\frac{\p}{\p
{w}}+(|z_1|^2-z_2\ov{z_1}+|z_1|^4)\frac{\p}{\p \ov{w}}.
\end{split} \end{equation*}
Thus
\begin{equation*} \begin{split}
[L,[L,\ov{L}]] &=\ov{z_1}\frac{\p}{\p \ov{w}}+\frac{\p}{\p
{z_2}}+3z_1\ov{z_1}^2\frac{\p }{\p \ov{w}}+(\ov{z_2}-2z_1
\ov{z_1}^2)\frac{\p }{\p {w}}-\ov{z_1}\cdot |z_1|^2\frac{\p}{\p
w}\\&= \frac{\p}{\p {z_2}}+(\ov{z_2}-3\ov{z_1}|z_1|^2)\frac{\p }{\p
{w}}+(\ov{z_1}+3\ov{z_1}|z_1|^2)\frac{\p }{\p \ov{w}}.
\end{split} \end{equation*}
Furthermore, we obtain
\begin{equation*} \begin{split}
[L,[L,[L,\ov{L}]]] &=3\ov{z_1}^2\frac{\p }{\p \ov{w}}-3\ov{z_1}^2\frac{\p }{\p {w}}=O(|z|^2).\\
[\ov{L},[L,[L,\ov{L}]]]) &=7|z_1|^2\frac{\p }{\p
\ov{w}}-7|z_1|^2\frac{\p }{\p {w}}=O(|z|^2).
\end{split} \end{equation*}
Notice that
$$
[\ov{L},[\ov{L},[L,\ov{L}]]]=-\ov{[L,[L,[L,\ov{L}]]] }=O(|z|^2),\ \
[\ov{L},[{L},[L,\ov{L}]]]=-\ov{[L,[\ov{L},[L,\ov{L}]]] }=O(|z|^2).
$$
Hence, it is clear that $t_L(L,0)\geq 6$. On the other hand,
$$
[\ov{L},[\ov{L},[L,[L,[L,\ov{L}]]]]]=6\frac{\p }{\p
\ov{w}}-6\frac{\p }{\p {w}}.
$$
It follows that $t_L(M,0)=6$.

\section{Order of vanishing along vector fields}

We first denote in this section  by $(x_1,\cdots,x_n)$ for the
coordinates of  $\mathbb{R}^n$. Let $U\subset \RR^n$ be an open
subset and let $\{X_1,X_2,\cdots,X_r\}$ be a set of linearly
independent real-valued smooth vector fields over $U$. Write
$\mathcal D$ for the (real) submodule of the $C^{\infty}(U)$-module
${\mathcal X}_{\RR}$ consisting of real-valued smooth vector fields
in $U$ generated by $\{X_1,X_2,\cdots,X_r\}$.

\begin{defn} Let $f\in C^\infty(p_0)$ be  a germ of (complex-valued)
smooth function at $p_0\in U$.

\noindent (1).  We say $\mathcal{D}^k(f)(p_0)=0$ if
$Y_k\left(Y_{k-1}(\cdots Y_1(f)\cdots)\right)(p_0)=0$ with each
$Y_j\in \mathcal{D}$ for $j=1,\cdots,k$.

\noindent (2). We say the order  of vanishing of $f$  along
$\mathcal D$, denoted by ${\nu}_{\mathcal{D}}(f)(p_0)$, to be $m\in
\NN\cup \{\infty\}$ if $\mathcal{D}^k(f)(p_0)=0$ for any $k<m$ and
if $\mathcal{D}^m(f)(p_0)\not =0$, namely, $Y_m\left(\cdots(
Y_1(f))\cdots\right)(p_0)\not = 0$ for a certain $Y_1,\cdots, Y_m$
with each $Y_j\in {\mathcal D}$ when $m<\infty$.
\end{defn}

Our main goal in this section is to prove the following
\begin{thm}\label{cor1.3} Let $f,g$ be germs of smooth functions at
$p_0$. Then

\noindent   (1). $${\nu}_{\mathcal D}(fg)(p_0)={\nu}_{\mathcal
D}(f)(p_0)+
  {\nu}_{\mathcal D}(g)(p_0).$$


\noindent (2). If $0\leq f\leq g$ and $g(p_0)=f(p_0)$, then $
{\nu}_{\mathcal D}(f)(p_0)\geq {\nu}_{\mathcal D}(g)(p_0)$.

\noindent (3). If $f\ge 0$, then $ {\nu}_{\mathcal D}(f)(p_0)$ is an
even number or infinity.
\end{thm}

 After a linear change of
coordinates, assume that $p_0=0$ and
\begin{equation} \begin{split}
X_j=\frac{\p}{\p x_j}+\widehat{X_j}\ \ \text{with} \
\widehat{X_j}|_{p_0}=0.
\end{split} \end{equation}

Write
$$E_{0}=\h{span}\{X_j|_{p_0},\ j=1,\cdots,r\}=\h{span}\{\frac{\p}{\p x_1}|_{p_0},\cdots,
\frac{\p}{\p x_r}|_{p_0}\}.$$

Define $E_1\subset T_{p_0}\Omega$ to be the linear span of $E_0$ and
the values at $p_0=0$ of  commutators of vector fields from
$\mathcal{D}$ of length $m_1$, where  $m_1$ is the smallest integer,
if existing, such that $E_1\supsetneq E_0$. $m_1$ is called the
first H\"omander number of $\mathcal{D}$. $\ell_1=\text{dim} E_1-r$
is called the multiplicity of $m_1$.

When $m_1<\infty$, we define $E_2\subset T_{p_0}\Omega$ to be the
linear span of $E_1$ and the values at $p_0=0$ of  commutators of
vector fields from $\mathcal{D}$ of length $m_2$, where  $m_2$ is
the smallest integer, if existing, such that $E_2\supsetneq E_1$.
$m_2$ is called the second H\"omander number of $\mathcal{D}$.
$\ell_2=\text{dim} E_1-\ell_1$ is called the multiplicity of $m_2$.

Inductively one can define a finite sequence of natural numbers
$m_1<\cdots<m_h$ with $h\ge 1$, and the associated linear subspaces
of $T_{p_0}U$: $E_1\varsubsetneq E_2\varsubsetneq \cdots
\varsubsetneq E_h\subset T_{p_0}U$. Here $m_h$ is the largest
H\"ormander number and
$\ell_{j+1}=\text{dim}E_{j+1}-\text{dim}E_{j}$ is the multiplicity
of $m_{j+1}$ for $j=0,\cdots,h-1$ . When $E_h=T_{p_0}U$, we call
$\mathcal{D}$ is of finite H\"ormander type at $p_0$. Otherwise, we
say $\mathcal{D}$ is of infinite type at $p_0$. A deep theorem on
the normalization  of $\mathcal{D}$ by Helffer-Nourrigat and
Boauendi-Rothschild give the following statements: (For a detailed
proof, see, for example, \cite[Theorem 3.5.2]{BER}):
\\

(1). There exists a coordinate system $(x,s,s')$ of $\RR^n$ centered
at $p_0$ which corresponds $p_0$ to $0$, where $x=(x_1,\cdots,x_r)$,
$s=(s_1,\cdots,s_h)$, $s_j=(s_{j1},\cdots,s_{j\ell_j})\in
\RR^{\ell_j}$ ($j=1,\cdots,h$), $s'=(s'_1,\cdots,s'_{m'})\in
\RR^{m'}$, such that in the new coordinates $(x,s,s')$, the
following holds:
\begin{equation}\label{001}\begin{split}
X_k=\frac{\p}{\p
x_k}+\sum\limits_{j=1}^{h}\sum\limits_{q=1}^{\ell_j}P_{k,j,q}(x,s_1,\cdots,s_{j-1})
\frac{\p}{\p
s_{jq}}+O_{\h{wt}}(0),\ \ k=1,\cdots,r.
\end{split} \end{equation}
Here after assigning the weight of $x:=(x_1,\cdots,x_r)$ to be one
and that of $s_j$ to be $m_j$ for $j=1,\cdots,h$, namely, after
setting $\h{wt}(x)=1$ and $\h{wt}(s_j)=m_j$ for $j=1,\cdots,h$, each
$P_{k,j,q}(x,s_1,\cdots,s_{j-1})$ is a weighted homogeneous
polynomial of weighted degree $m_j-1$ for  $j=1,\cdots,h$. We define
the weight of $s'$ to be infinity. For each $x_k,s_{jq}$, set
\begin{equation}\begin{split}
\text{wt}(\frac{\p}{\p x_k})=-1,\ \ \text{wt}(\frac{\p}{\p
s_{jq}})=-m_j.
\end{split} \end{equation}
Write
\begin{equation}\begin{split}
X^0_k=\frac{\p}{\p
x_k}+\sum\limits_{j=1}^{h}\sum\limits_{q=1}^{\ell_j}P_{k,j,q}(x,s_1,\cdots,s_{j-1})\frac{\p}{\p
s_{jq}}.
\end{split} \end{equation}
Then $X_k^0$ is a homogeneous vector field of weighted degree $-1$
for each $k$.  The meaning $O_{\h{wt}}(0)$  in (\ref{001}) is that
each term in the weighted  formal Taylor expansion of $X_k-X_k^0$ at
$0$ has weight at least $0$ (after  assigning the weight of $s'$ to
be infinity).

\medskip

(2). Let $\{Y_{1,1},\cdots,Y_{1,\ell_1}\}$ be the commutators of
vector fields from $\mathcal{D}$ of length $m_1$ such that
\begin{equation}\begin{split}
E_1=E_0\oplus \h{span}\{Y_{1,1}|_0,\cdots,Y_{1,\ell_1}|_0\}.
\end{split} \end{equation}
Notice by the statement in (1) that the lowest degree in the formal
expansion of each $Y_{1,j}$  at $0$ is $-m_1$. Hence
\begin{equation}\begin{split}
\h{span}\{Y_{1,1},\cdots,Y_{1,l_1}\}|_0=\h{span}\{\frac{\p}{\p
s_{1,1}}|_0,\cdots,\frac{\p}{\p s_{1,\ell_1}}|_0\}.
\end{split} \end{equation}
After a linear change of coordinates in
$(s_{1,1},\cdots,s_{1,l_1})$, we may assume that
\begin{equation}\begin{split}
Y_{1,j}|_0=\frac{\p}{\p s_{1,j}}|_0\ \text{for}\ j=1,\cdots,\ell_1.
\end{split} \end{equation}
Inductively, we then define for $k\le h$  the commutators of vector
fields from $\mathcal{D}$ of length $m_k$
\begin{equation}\begin{split}
\{Y_{k,j}\}_{j=1}^{\ell_k}\ \text{such that}\
Y_{k,j}|_0=\frac{\p}{\p s_{k,j}}|_0\ \text{for}\ j=1,\cdots,\ell_k\
\text{with \ \ wt}(Y_{k,j})=O_{\h{wt}}(-m_k).
\end{split} \end{equation}
Then
\begin{equation}\begin{split}
Y_{k,j}=\frac{\p}{\p
s_{k,j}}+\sum\limits_{\tau>k}\sum_{q=1}^{\ell_\tau}Q_{k,j,\tau,q}(x,s_1,\cdots,s_{k-1})\frac{\p}{\p
s_{\tau q}}+O_{wt}(-m_k+1),\ \ j=1,\cdots,\ell_k.\\
 Y^0_{k,j}=\frac{\p}{\p
s_{k,j}}+\sum\limits_{\tau>k}\sum_{q=1}^{\ell_\tau}Q_{k,j,\tau,q}(x,s_1,\cdots,s_{k-1})
\frac{\p}{\p
s_{\tau q}}
\end{split} \end{equation}
Here $Q_{k,j,\tau,q}(x,s_1,\cdots,s_{k-1})$ are weighted homogeneous
polynomials of degree $m_\tau-m_k$.  The notation  $O_{wt}(-m_k+1)$
has a  similar meaning. Write $Y_{0,j}=X_j$.

Write $Y_k=(Y_{k1},\cdots,Y_{k\ell_k})$ and
$Y_k^{\b_k}=Y_{k1}^{\a_{k,1}}\cdots Y_{k,\ell_k}^{\a_{k,\ell_k}}$
with $\b_k=(\a_{k,1},\cdots,\a_{k,\ell_k})$ for $k=0,\cdots,h$.

\medskip With   the just discussed normalization of $\mathcal D$ at our disposal,  we now
let  $f\in C^{\infty}(U)$ and let $f\sim\sum\limits_{j\geq
N}^{\infty}f^{(j)}$ be the weighted formal  expansion
 with $wt(s')=\infty$. Write, in what follows,
$\prod_{j=0}^{m_h}Y^{\b_{j}}_j=Y^{\b_{0}}_0\cdots Y^{\b_{m_h}}_{m_h}
$.
Then
\begin{equation}\label{002}\begin{split}
\left(\prod_{k=0}^{m_h}
Y_k^{\beta_k}\right)(f)=\left(\prod_{k=0}^{m_h}
(Y_k^{0})^{\beta_k}\right)(f^{(N)})+O_{wt}(N^*+1).
\end{split} \end{equation}
Here $\left(\prod_{k=1}^{\ell} (Y_k^{0})^{\beta_k}\right)(f^{(N)})$
is a weighted homogeneous polynomial of degree
$N^*:=N-\sum_{k=0}^{m_h} \beta_k m_k$, where $m_0=r$.


Let $\Gamma=x^{\b_0} s_1^{\beta_1}\cdots s_h^{\beta_h}$. Then
$\Gamma$ is a polynomial of weighted degree $m:=\sum_{k=0}^{h}|\b_k|
m_k.$ By (\ref{002}), $\nu_{\mathcal D}(\G)(0)\ge m$. On the other
hand $Y_0^{\b_0}Y^{\b_1}_1\cdots Y^{\b_h}_h(\Gamma)(0)=\b_0 !\b_1
!\cdots \b_h!\neq 0.$ Since $Y_0^{\b_0}Y^{\b_1}_1\cdots
Y^{\b_h}_h(\Gamma)(0)$ is a finite linear combination of terms of
the form:
$$Z_m\cdots Z_1(\G)(0),\ \ Z_j\in \{X_1,\cdots, X_r\}\ \h{ for } j=1,\cdots,r $$
we conclude that ${\mathcal D}^m(\G)(0)\not =0$. Thus we conclude
that
\begin{equation}\label{003}\begin{split}
 {\nu}_{\mathcal D}(\G)(0)=m=\sum_{k=0}^{h}\b_k m_k.
\end{split} \end{equation}

\begin{thm}\label{thm1.2} Assume the notations and definitions set
up above. Let $f$ be a germ of complex-valued smooth function at
$p_0=0$. Assume that $f= f^{(m)}+O_{wt}(m+1)$, where
  $f^{(m)}$ is the  lowest  weighted non-zero homogeneous term of  degree $m\in \NN$.
  Then $m= {\nu}_{\mathcal D}(f)(0)$.
   If $f= O_{wt}(m)$ for any $m$, then ${\nu}_{\mathcal D}(f)(0)=\infty.$
\end{thm}

\begin{proof}
  Assume that
$$
f^{(m)}=\sum\limits_{\sum\limits_{j=0}^{h}|\beta_j|m_j=m}a_{\beta_0\cdots
\beta_h}x^{\beta_0}s_1^{\beta_1}\cdots s_h^{\beta_h},\ \
a_{\beta_0\cdots \beta_h}\not \equiv 0.
$$
We first find the largest $\b_h$ in the lexicographic order among
non-zero monomials. Then among those non-zero  monomials with $\b_h$
being maximal, we find the largest $\b_{h-1}$  in the lexicographic
order. By an induction, we get the $(\beta_h,\cdots,\beta_0)$ which
is lexicographically maximal among those non-zero monomials in the
above expansion of $f^{(m)}$.
Then by (\ref{003}),
$$
Y_0^{\b_0}Y^{\b_1}_1\cdots Y^{\b_h}_h f(p_0)=a_{\beta_0\cdots
\beta_h}\b_h!\cdots \b_0!\neq 0.
$$
As in the argument to prove (\ref{003}), we conclude that
${\nu}_{\mathcal D}(f)(0)=m$.
\end{proof}

\begin{proof}[Proof of Theorem \ref{cor1.3}] We can assume $p_0=0$ and ${\mathcal D}$
 has been normalized as above.
(1) is a direct consequence of Theorem \ref{thm1.2}. As for (2), we have
 $g-f\geq 0$ and $(g-f)(p_0)=0$. Suppose that  $
{\nu}_{\mathcal D}(f)(p_0)<  {\nu}_{\mathcal D}(g)(p_0)$. Then
$g-f=-f^{(m)}+O_{wt}(m+1)$, where $f^{(m)}$ is the lowest non-zero
weighted homogeneous polynomial in the expansion of $f$. Apparently,
$f^{(m)}\geq 0$ as $f\ge 0$ and ${\nu}_{\mathcal D}(f)(p_0)<\infty$.
Also, $-f^{(m)}\geq 0$ as $f-g\ge 0$ and $f^{(m)}\not\equiv 0$. This
is a contradiction.

To prove (3), assume that ${\nu}_{\mathcal D}(f)(p_0)=m$ with
$0<m<\infty$. Then $f^{(m)}\not \equiv 0$ but $f^{(m)}\ge 0$. Since
$f^{(m)}(\tau x, \tau^{m_1}s_1,\cdots,
\tau^{m_h}s_h)=\tau^{m}f^{(m)}(x, s_1,\cdots, s_h)\ge 0$ for any
$\tau\in \RR$. It follows that $m$ is an even number.
\end{proof}

The following result reducing the vanishing order along a system of
vector fields to  that along a single one could be very useful in
applications, which, for instance, also gives an immediate proof of
Theorem \ref{cor1.3}.
\begin{thm}\label{imp}
  Let $L(x_1,\cdots,x_r)=\sum\limits_{j=1}^{r}a_jx_j$ be a non-zero
  linear function in $x=s_0$. Let
  $$
  Z=\sum\limits_{j=0}^{h}\sum\limits_{q=1}^{\ell_j}
\frac{1}{m_j}t_{j,q}L(x)^{m_j-1}Y_{j,q},\ \ t_{j,q}\in \mathbb{R}.
  $$
  Let $f\in C^\infty(p_0)$. For a generic choice of
  $\{t_{j,q}\}$,
  $$
{\nu}_{\mathcal{D}}(f)(p_0)= {\nu}_{Z}(f)(p_0).$$ Here
${\nu}_{Z}(f)(p_0)$ is the smallest non-negative integer $\ell$ such
that $Z^\ell(f)(p_0)\not = 0$ or $\infty$ when such an $\ell$ does
not exist.
\end{thm}

\begin{proof}
  Apparently, $$ {\nu}_{{Z}}(f)(p_0)\geq
{\nu}_{\mathcal{D}}(f)(p_0):=m.$$ We need only to  assume that
$m<\infty$ and prove  that $Z^mf^{(m)}(p_0)\neq 0$ for a generic
choice of
$$(t_0,t_1,\cdots,t_h):=(t_{0,1},\cdots,t_{h,\ell_h})\in \RR^{n_0},$$
namely, for any choice of $(t_{0,1},\cdots,t_{h,\ell_h})$  but away
from a certain proper real analytic variety in $\RR^{n_0}$. (Here
$n_0=r+\sum_{j=1}^h\ell_j.$)

Let $\gamma(\tau)$ be the integral curve of $Z$ through $p_0=0$.
Namely, $\frac{d \gamma(\tau)}{d\tau}=Z\circ \gamma(t)$ and
$\gamma(0)=p_0=0$. Write
$\gamma(\tau)=(x(\tau),s_1(\tau),\cdots,s_h(\tau),\wt{s(\tau)})$.
Then $\frac{d x(\tau)}{d\tau}=t_0+O(\tau)$. Hence
$x(\tau)=t_0\tau+O(\tau^2)$. Similarly,
$$
\frac{d
s_j(\tau)}{d\tau}=\frac{1}{m_j}t_jL^{(m_j-1)}(t_0)\tau^{m_j-1}+O(\tau^{m_j}).
$$
Hence
$$
s_j(\tau)=t_jL^{m_j-1}(t_0)\tau^{m_j}+O(\tau^{m_j+1}),\ \
\wt{s}(\tau)=O(\tau^{m+1}).
$$
Now,
$$
Z^m(f)(0)=\frac{d}{d t^m}(f^{(m)}\circ \gamma(t))|_{t=0}.
$$
Write
$$
f^{(m)}(s_0,s_1,\cdots,s_h)=\sum a_{\beta_0\beta_1\cdots
\beta_h}s_0^{\beta_0}\cdots s_h^{\beta_h}.
$$
Then
$$
Z^m(f)(p_0=0)=\sum a_{\beta_0\beta_1\cdots
\beta_h}\left(L(t_0)\right)^{\sum\limits_{j=0}^{h}(m_j-1)|\beta_j|}m!t_0^{\beta_0}\cdots
t_h^{\beta_h}.
$$
Suppose $Z^m(f)(p_0)\equiv 0$ for any choice of
$(t_0,t_1,\cdots,t_h)$. Then
$$
\sum a_{\beta_0\beta_1\cdots
\beta_h}L^{m-\sum_{j=0}^h\b_j}(t_0)m!t_0^{\beta_0}\cdots
t_j^{\beta_j}\equiv 0,\ (t_0,t_1,\cdots,t_h)\in \RR^{n_0}.
$$
Hence $a_{\beta_0\beta_1\cdots \beta_h}\equiv 0.$ This contradicts
that $f^{(m)}\equiv 0$. Hence for $(t_{0,1},\cdots,t_{h,\ell_h})$
not from the proper real analytic subset defined by $$\sum
a_{\beta_0\beta_1\cdots
\beta_h}L^{m-\sum_{j=0}^h\b_j}(t_0)m!t_0^{\beta_0}\cdots
t_j^{\beta_j}= 0,$$  then
$$
{\nu}_{\mathcal{D}}(f)(p_0)= {\nu}_{Z}(f)(p_0)=m.$$
\end{proof}


\subsection{Application in the study of regular  types of a real hypersurfaces}

We  now adapt the notations and definitions which we have set up in
Section \ref{subb1}. We let $M\subset \CC^n$ be a smooth real
hypersurface and let $B$ be a smooth subbundle  of complex dimension
$s$ of $T^{(1,0)}M$. Write  ${\mathcal D}_B$ for the real submodule
over real-valued smooth functions over $M$ that is generated by
$\h{Re}(L)$ and $\h{Im}(L)$ for any $L\in  \Gamma_\infty(B)$, where
$\Gamma_\infty(B)$ denotes the set of smooth sections of $B$. Then
 ${\mathcal D}_B$ is locally  generated by $2s$ $\RR$-linearly independent real vector
fields.  For any   smooth function $f$, we define
$$\nu_B(f)(p)=\nu_{{\mathcal D}_B}(f)(p).$$
Notice that $\nu_B(f)(p)$ is the least $m$ such that $Y_m\cdots
Y_1(f)(p)\not =0$ where $Y_j\in {\mathcal M}_1(B)$. If such an $m$
does not exist, $\nu_B(f)(p)=\infty.$

 We then apparently have
\begin{equation}\label{006}\min_{L\in \Gamma_\infty(B),\ L|_p\not
=0}\nu_{B}(\ld(L,L))(p)=c^{(s)}(B,p).
\end{equation}
 When $B=\h{span}\{L\}$, as before, we
write $\nu_L(f)(p)=\nu_B(f)(p).$ Now, we can reformulate Theorem
\ref{cor1.3} as follows:
\begin{thm}\label{cor2.3} Let $B$ and $M$ be as just defined.
Let $f,g$ be germs of (complex-valued) smooth functions over $M$ at
$p\in M$. Then

\noindent   (1). $${\nu}_{B}(fg)(p)={\nu}_{B}(f)(p)+
  {\nu}_{B}(g)(p).$$


 \noindent (2). If $f,g$ are real-valued and $0\leq f\leq g$ with
$g(p)=f(p)$, then $ {\nu}_{B}(f)(p)\geq {\nu}_{B}(g)(p)$.

\noindent (3). If $f\ge 0$, then $ {\nu}_{B}(f)(p)$ is an even
number if not infinity.
\end{thm}

\begin{cor}\label{cor2.4} Let  $M\subset \CC^n$ be a real hypersurface with $n\ge 2$.
Let $L$ be a CR vector field along $M$ not vanishing at any point.
Let $f,g$ be germs of smooth functions over $M$ at $p\in M$. Then

\noindent   (1). $${\nu}_{L}(fg)(p)={\nu}_{L}(f)(p)+
  {\nu}_{L}(g)(p).$$

\noindent (2). If $0\leq f\leq g$ and $g(p)=f(p)$, then $
{\nu}_{L}(f)(p)\geq {\nu}_{L}(g)(p)$.
\end{cor}

Corollary \ref{cor2.4} had  first appeared in a paper of D'Angelo
[pp 105, 3. Remark, \cite{DA3}].

\begin{cor}\label{cor1.5}
  Assume that $M$ is pseudoconvex. For any $p\in M$ and any subbundle  $B$
   of $T^{(1,0)}M$ with complex dimension $s$, then

\noindent  (1). $c^{(s)}(B,p)$ and   $c^{(s)}(M,p)$ are  even
numbers if not infinity.

\noindent  (2). Let $L$ be a CR  vector field of $M$ not equal to
zero at any point. If $\mathcal{D}_{L}^{2j}\lambda(X,{X})(0)=0$ and
  $\mathcal{D}_{L}^{2k}\lambda(Y,{Y})(0)=0$, then
$\mathcal{D}_{L}^{j+k+1}\lambda(X,{Y})(0)=0$.
\end{cor}
\begin{proof} Pseudo-convexity of $M$ implies that $\ld(L,L)\ge 0$.
Then (1) follows immediately from (\ref{006}) and Theorem
\ref{cor2.3}(3).


 To prove (2),  by (1),  the assumption in (2) shows that
$\mathcal{D}_{L}^{2i+1}\lambda(X,{X})(0)=0$ and
  $\mathcal{D}_{L}^{2j+1}\lambda(Y,{Y})(0)=0$. From the Schwarz
  inequality for a non-negative Hermition form, it follows that
$$|\lambda(X,{Y})|^2\leq \lambda(X,{X})\cdot
\lambda(Y,{Y}).$$ The statement in (2) follows from Theorem
\ref{cor2.3}.
\end{proof}

\bigskip
Still let $M, \ B$ be as above and assume that $M$ is pseudoconvex.
 For
$V_B=\{L_1,\cdots,L_s\}$, a basis of smooth sections  of $B$ near
$p$. Let $L$ be the one such that $c^{(s)}(B,
p)=\nu_{B}\left(\ld(L,L)\right)(p)$ and $L=\sum_{j}a_jL_j$. Then
$\ld(L,L)=\sum_{j,k}a_j\ov{a_k}\ld(L_j,L_k).$ Define the trace of
the Levi-form  along $V_B$ by
\begin{equation}
  \text{tr}_{V_B}\ld_M(q)=\sum\limits_{j=1}^{s}  \ld(L_j,{L_j}),
  \ \ q\approx p.
 \end{equation}
By Theorem \ref{cor2.3} and Corollary \ref{cor1.5}(2), it follows
that
$$\nu_{B}\left(\ld(L,L)\right)=\min_{1\leq j\leq s}\{\nu_{B}\left(\ld(L_j,L_j)\right)\}=
\nu_{B}\left( \text{tr}_{V_B}\ld_{M}\right)(p).$$ Hence, we have
\begin{cor}
Assume that $M$ is pseudoconvex and $B$ is a smooth complex
subbundle of $T^{(1,0)}M$. For any local linearly independent local
frame $V_B$ of $\Gamma_\infty(B)$ near $p\in M$, it holds that
$$c^{(s)}(B,p)=\nu_{B}\left( \text{tr}_{V_B}\ld_{M}\right)(p).$$
\end{cor}

\noindent Xiaojun Huang, Department of Mathematics, Rutgers
University, New Brunswick, NJ 08903, USA
(huangx$@$math.rutgers.edu);

\noindent Wanke Yin,   School of Mathematics and Statistics, Wuhan
University, Wuhan, Hubei 430072, China (wankeyin@whu.edu.cn).

\end{document}